\documentclass[12pt, twoside, leqno]{article}

\usepackage{amsmath,amsthm}
\usepackage{amssymb}

\usepackage{enumerate}
\usepackage{mathrsfs}
\usepackage{paralist}

\usepackage{color}

\usepackage{graphicx}

\theoremstyle{definition}

\numberwithin{equation}{section}

\renewcommand{\setminus}{\smallsetminus}
\usepackage{ifthen}

\newcommand{\s}{\sigma}
\newcommand{\Rad}{{\mathrm{\bf Rad}}}
\newcommand\remove[1]{}
\renewcommand{\hat}{\widehat}
\renewcommand{\c}{{\{-1,1\}^n}}
\newcommand{\rnote}[1]{}
\newcommand{\jnote}[1]{}

\renewcommand{\d}{\delta}

\newcommand{\n}{\{1,\ldots,n\}}
\def\cprime{$'$}
\newcommand{\1}{\mathbf{1}}
\newcommand{\e}{\varepsilon}
\newcommand{\R}{\mathbb{R}}
\renewcommand{\P}{\mathbb{P}}

\newcommand{\N}{\mathbb{N}}

\renewcommand{\subset}{\subseteq}
\newtheorem{theorem}{Theorem}[section]
\newtheorem{lemma}[theorem]{Lemma}

\newtheorem{corollary}[theorem]{Corollary}

\newtheorem{definition}[theorem]{Definition}

\newtheorem{remark}{Remark}[section]

\newcommand{\eqdef}{\stackrel{\mathrm{def}}{=}}
\date{}

\renewcommand{\le}{\leqslant}
\renewcommand{\ge}{\geqslant}
\renewcommand{\leq}{\leqslant}

\include{psfig}

\renewcommand{\epsilon}{\varepsilon}

\theoremstyle{remark}

\renewcommand{\phi}{\varphi}

\begin{document}


\baselineskip=17pt


\title{Pisier's inequality revisited}

\author{Tuomas Hyt\"onen\\
Department of Mathematics and Statistics\\
University of Helsinki\\
P.O.B. 68,
FI-00014 Helsinki, Finland\\
E-mail: tuomas.hytonen@helsinki.fi \and
Assaf Naor\\
Courant Institute of Mathematical Sciences\\
New York University\\
251 Mercer Street, New York NY, 10012, USA\\
E-mail: naor@cims.nyu.edu}

\date{}

\maketitle


\renewcommand{\thefootnote}{}



\renewcommand{\thefootnote}{\arabic{footnote}}
\setcounter{footnote}{0}


\begin{abstract}
Given a Banach space $X$, for $n\in \N$ and $p\in (1,\infty)$ we
investigate the smallest constant $\mathfrak P\in (0,\infty)$  for
which every $n$-tuple of functions $f_1,\ldots,f_n:\{-1,1\}^n\to X$
satisfies
\begin{multline*}
\int_{\{-1,1\}^n}\Bigg\|\sum_{j=1}^n
\partial_jf_j(\e)\Bigg\|^pd\mu(\e)\\\le
\mathfrak{P}^p\int_{\{-1,1\}^n}\int_{\{-1,1\}^n}\Bigg\|\sum_{j=1}^n
\d_j\Delta f_j(\e)\Bigg\|^pd\mu(\e) d\mu(\delta),
\end{multline*}
where $\mu$ is the uniform probability measure on the discrete
hypercube $\{-1,1\}^n$ and $\{\partial_j\}_{j=1}^n$ and
$\Delta=\sum_{j=1}^n\partial_j$ are the hypercube partial
derivatives and the hypercube Laplacian, respectively. Denoting this
constant by $\mathfrak{P}_p^n(X)$, we show that
$$\mathfrak{P}_p^n(X)\le \sum_{k=1}^{n}\frac{1}{k}$$ for every Banach
space $(X,\|\cdot\|)$. This extends the classical Pisier inequality,
which corresponds to the special case $f_j=\Delta^{-1}\partial_j f$
for some $f:\{-1,1\}^n\to X$. We show that $\sup_{n\in
\N}\mathfrak{P}_p^n(X)<\infty$ if either the dual $X^*$ is a
$\mathrm{UMD}^+$ Banach space, or for some $\theta\in (0,1)$ we have
$X=[H,Y]_\theta$, where $H$ is a Hilbert space and $Y$ is an
arbitrary Banach space. It follows that $\sup_{n\in
\N}\mathfrak{P}_p^n(X)<\infty$ if $X$ is a Banach lattice of
nontrivial type.
\end{abstract}

\section{Introduction}\label{sec:intro}

Fix a Banach space $(X,\|\cdot\|)$ and $n\in \N$. For every
$f:\{-1,1\}^n\to X$ and $j\in \{1,\ldots,n\}$ the hypercube $j$th
partial derivative of $f$, which is denoted $\partial_j f:\{-1,1\}^n\to X$,
is defined as
\begin{equation}
\partial_jf(\e)\eqdef \frac{f(\e)-f(\e_1,\ldots,\e_{j-1},-\e_j,\e_{j+1},\ldots,\e_n)}{2}.
\end{equation}
The hypercube Laplacian of $f$, denoted $\Delta f:\{-1,1\}^n\to X$,
is
\begin{equation}\label{eq:def Delta}
\Delta f(\e)\eqdef \sum_{j=1}^n \partial_jf(\e).
\end{equation}
It is immediate to check that $\Delta$ is invertible on the space of
all mean zero functions $f:\c\to X$. Below $\Delta^{-1}$ is
understood to be defined for every $f:\c\to X$ by setting
$\Delta^{-1}f=\Delta^{-1}\overline f$. Here  $\overline
f=f-\int_{\c}f(\d)d\mu(\d)$, where $\mu$ denotes the uniform
probability measure on $\{-1,1\}^n$.

The following inequality is due to Pisier~\cite{Pis86}. Throughout
this paper the asymptotic notation $\lesssim,\gtrsim$ indicates the
corresponding inequalities up to universal constant factors. We will
also denote equivalence up to universal constant factors by
$\asymp$, i.e., $A\asymp B$ is the same as $(A\lesssim B)\wedge
(A\gtrsim B)$.

\begin{theorem}[Pisier's inequality]\label{thm:pisier}
For every Banach space $(X,\|\cdot\|)$, every $n\in \N$, every $p\in
[1,\infty]$ and every $f:\{-1,1\}^n\to X$, we have
\begin{multline}\label{eq:classical pisier}
\left(\int_{\c}\Bigg\|f(\e)-\int_{\c}f(\d)d\mu(\d)\Bigg\|^pd\mu(\e)\right)^{1/p}\\
\lesssim \log n \left(\int_\c\int_\c\Bigg\|\sum_{j=1}^n \d_j
{\partial_j} f(\e)\Bigg\|^pd\mu(\e)d\mu(\d)\right)^{1/p}.
\end{multline}
\end{theorem}
Due to the application of Pisier's inequality to the theory of
nonlinear type (see~\cite{Pis86,NS02,GN10,Nao12}), it is of great
interest to understand when~\eqref{eq:classical pisier} holds true
with the $\log n$ term replaced by a constant that may depend on the
geometry of $X$ but is independent of $n$. Talagrand
proved~\cite{Tal93} that the $\log n$ term in~\eqref{eq:classical
pisier} is asymptotically optimal for general Banach spaces $X$,
Wagner proved~\cite{Wag00} that the $\log n$ term
in~\eqref{eq:classical pisier} can be replaced by a universal
constant if $p=\infty$ and $X$ is a general Banach space, and
in~\cite{NS02} it is shown that the $\log n$ term
in~\eqref{eq:classical pisier} can be replaced by a constant that is
independent of $n$ if $X$ is a $\mathrm{UMD}$ Banach space.
It remains an intriguing open question whether every Banach space of
nontrivial type satisfies~\eqref{eq:classical pisier} with the $\log
n$ term replaced by a constant that is independent of $n$. If true,
this would resolve a 1976 question of Enflo~\cite{Enf76} by
establishing that Rademacher type $p$ and Enflo type $p$ coincide
(see~\cite{NS02,Nao12} and Section~\ref{sec:enflo} below).

Here we obtain a new class of Banach spaces that satisfies a
dimension-independent Pisier inequality. Our starting point is the
following extension of Pisier's inequality.
\begin{definition}[Pisier constant of $X$]\label{def:pisier const}
The $n$-dimensional Pisier constant of $X$ (with exponent $p$),
denoted $\mathfrak{P}_p^n(X)$, is the infimum over those
$\mathfrak{P}\in (0,\infty)$ such that every $f_1,\ldots,f_n:\c\to
X$ satisfy

\begin{multline}\label{eq:def extended pisier}
\left(\int_\c\Bigg\|\sum_{j=1}^n
\partial_jf_j(\e)\Bigg\|^pd\mu(\e)\right)^{1/p}\\\le
\mathfrak{P}\left(\int_\c\int_\c\Bigg\|\sum_{j=1}^n \d_j\Delta
f_j(\e)\Bigg\|^pd\mu(\e)d\mu(\d)\right)^{1/p}.
\end{multline}
We also set
$$
\mathfrak{P}_p(X)\eqdef \sup_{n\in \N} \mathfrak{P}_p^n(X).
$$
\end{definition}
Inequality~\eqref{eq:def extended pisier} reduces to Pisier's
inequality if we choose $f_j=\Delta^{-1}\partial_j f$ for some
$f:\c\to X$. The generalized inequality~\eqref{eq:def extended
pisier} has the advantage of being well-behaved under duality, as
explained in Section~\ref{sec:dual}. The following theorem yields a
logarithmic bound on $\mathfrak{P}_n^p(X)$, thus extending Pisier's
inequality.

\begin{theorem}\label{thm:generalized pisier}
For every Banach space $X$, every $p\in [1,\infty]$ and every $n\in
\N$,
$$
\mathfrak{P}_p^n(X)\le \sum_{k=1}^{n}\frac{1}{k}.
$$
\end{theorem}
Our approach yields a quantitative improvement over Pisier's
inequality only in lower order terms: an optimization of Pisier's
argument (as carried out in~\cite{Nao12}) shows that the $O(\log n)$
term in~\eqref{eq:classical pisier} can be taken to be at most $\log
n+ O(\log\log n)$, while Theorem~\ref{thm:generalized pisier} shows
that this term can be taken to be $\log n+ O(1)$.

In~\cite{NS02} it was shown that the logarithmic term
in~\eqref{eq:classical pisier} can be replaced by a constant that is
independent of $n$ if $X$ is a $\mathrm{UMD}$ Banach space. Recall
that $X$ is a $\mathrm{UMD}$ Banach space if for every $p\in
(1,\infty)$ there exists a constant $\beta\in (0,\infty)$ such that
if $\{M_j\}_{j=0}^n$ is a $p$-integrable $X$-valued martingale
defined on some probability space $(\Omega,\P)$, then for every
$\e_1,\ldots,\e_n\in \{-1,1\}$ we have
\begin{equation}\label{eq:def UMD}
\int_\Omega \Bigg\|M_0+\sum_{j=1}^n \e_j(M_j-M_{j-1})  \Bigg\|^pd\P\le \beta^p \int_\Omega \|M_n\|^pd\P.
\end{equation}
The infimum over those $\beta\in (0,\infty)$ for which~\eqref{eq:def
UMD} holds true is denoted $\beta_p(X)$. It can be shown
(see~\cite{Bur86}) that $\beta_p(X)\lesssim
\frac{p^2}{p-1}\beta_2(X)$, so in order to define the $\mathrm{UMD}$
property it suffices to require the validity of~\eqref{eq:def UMD}
for $p=2$. $\mathrm{UMD}$ Banach spaces are known to be
superreflexive~\cite{Mau75,Ald79}, and one also has
$\beta_{q}(X^*)=\beta_p(X)$, where $q=p/(p-1)$ (see e.g.
\cite{Bur86}).

In~\cite{Gar90} Garling investigated the natural weakening
of~\eqref{eq:def UMD} in which the desired inequality is required to
hold true in expectation over $\e_1,\ldots,\e_n\in \{-1,1\}$ rather
than for every $\e_1,\ldots,\e_n\in \{-1,1\}$. Specifically, say
that $X$ is a $\mathrm{UMD}^+$ Banach space if for every $p\in
(1,\infty)$ there exists a constant $\beta\in (0,\infty)$ such that
if $\{M_j\}_{j=0}^n$ is a $p$-integrable $X$-valued martingale
defined on some probability space $(\Omega,\P)$ then
\begin{equation}\label{eq:def UMD+}
\int_\c\int_\Omega \Bigg\| M_0+\sum_{j=1}^n \e_j(M_j-M_{j-1})
\Bigg\|^pd\P d\mu(\e)\le \beta^p \int_\Omega \|M_n\|^pd\P.
\end{equation}
The infimum over those $\beta$ for which~\eqref{eq:def UMD+} holds
true is denoted $\beta_p^+(X)$.

\begin{theorem}\label{thm:UMD+}
If $X$ is a Banach space such that $X^*$ is $\mathrm{UMD}^+$ then
the following inequality holds true. Fix $p\in (1,\infty)$ and $n\in
\N$. For every function $F:\c\times\c\to X$ and $j\in \n$ denote
$$F_j(\e)\eqdef\int_\c \d_jF(\e,\d)d\mu(\d).$$ Then
\begin{multline}\label{eq:umd+ F}
\left(\int_\c\Bigg\|\sum_{j=1}^n
\Delta^{-1}\partial_jF_j(\e)\Bigg\|^pd\mu(\e)\right)^{1/p}\\\le
\beta_q^+(X^*)\left(\int_\c\int_\c\|F(\e,\d)\|^pd\mu(\e)d\mu(\d)\right)^{1/p},
\end{multline}
where $q=p/(p-1)$.
\end{theorem}
For every $f_1,\ldots,f_n:\c\to X$, an application of
Theorem~\ref{thm:UMD+} to the function $F(\e,\d)=\sum_{j=1}^n \d_j
f_j(\e)$ yields the following estimate on the Pisier constant of a
$\mathrm{UMD}^+$ Banach space.
\begin{corollary}\label{cor:UMD+}
$ \mathfrak{P}_{p}(X)\le \beta_{q}^+(X^*). $
\end{corollary}

It is unknown if a $\mathrm{UMD}^+$ Banach space must also be a
$\mathrm{UMD}$ Banach space, though it seems reasonable to
conjecture that there are $\mathrm{UMD}^+$ spaces that are not
$\mathrm{UMD}$. Regardless of this, Theorem~\ref{thm:UMD+} and
Corollary~\ref{cor:UMD+} are conceptually different from the result
of~\cite{NS02}, which relies on the full force of the $\mathrm{UMD}$
condition, i.e. it requires the validity of~\eqref{eq:def UMD} for
every choice of signs $\e_1,\ldots,\e_n$, while our argument needs
such estimates to hold true only for an average choice of signs. We
also have a quantitative improvement: in~\cite{NS02} it was shown
that Pisier's inequality holds true with the $O(\log n)$ term
in~\eqref{eq:classical pisier} replaced by
$\beta_p(X)=\beta_q(X^*)$, while we obtain the same estimate with
the $O(\log n)$ term in~\eqref{eq:classical pisier} replaced by
$\beta_q^+(X^*)\le \beta_q(X^*)$. Geiss proved~\cite{Gei99} that for
every $\eta\in (0,1)$ there is $C_\eta\in (0,\infty)$ such that for
every $M>1$ there is  a Banach space $X$ that satisfies
$$
\infty>\beta_q(X^*)\ge C_\eta\beta_q^+(X^*)^{2-\eta}\ge M.
$$

\begin{remark}{\em Inequality~\eqref{eq:umd+ F} is an extension of the generalized
Pisier inequality~\eqref{eq:def extended pisier}, but for general
Banach spaces it behaves very differently: unlike the logarithmic
behavior of Theorem~\ref{thm:generalized pisier}, the best constant
appearing in the right hand side of~\eqref{eq:umd+ F} for a general
Banach space $X$ must be at least a constant multiple of $\sqrt{n}$,
as exhibited by the case $X=L_1((\c,\mu),\R)$ and $F:\c\times \c\to
X$ given by $F(\e,\d)(\eta)=\prod_{i=1}^n (1+\e_i\d_i\eta_i)$.}
\end{remark}

Suppose that $\theta\in (0,1)$ and $X=[H,Y]_\theta$, where $H$ is a
Hilbert space and $Y$ is an arbitrary Banach space. Here
$[\cdot,\cdot]_\theta$ denotes complex interpolation
(see~\cite{BL76}). Theorem~\ref{thm:interpolation} below shows that
in this case $\mathfrak{P}_p(X)<\infty$, and therefore Pisier's
inequality holds true with the $\log n$ term in~\eqref{eq:classical
pisier} replaced by a constant that is independent of $n$. Pisier
proved~\cite{Pis79} that every Banach lattice of nontrivial type
(see~\cite{LT79}) is of the form $[H,Y]_\theta$ for some $\theta\in
(0,1)$, so we thus obtain the desired dimension independence in
Pisier's inequality for Banach lattices of nontrivial type. This
result does not follow from previously known cases in which a
dimension-independent Pisier inequality has been proved, since, as
shown by Bourgain~\cite{Bou83,Bou84}, there exist Banach lattices of
nontrivial type which are not $\mathrm{UMD}$. Note, however, that we
are still far from proving the conjectured dimension-independent
Pisier inequality for Banach spaces with nontrivial type: any space
of the form $[H,Y]_\theta$ admits an equivalent norm whose modulus
of smoothness has power type $2/(1+\theta)$ (see~\cite{Pis79,CR82}),
while there exist Banach spaces with nontrivial type that do not
admit such an equivalent norm (see~\cite{Jam74,JL75,Jam78,PX87}).

\begin{theorem}\label{thm:interpolation}
Let $X,Y$ be Banach spaces and let $H$ be a Hilbert space. Suppose
that for some $\theta\in (0,1)$ we have $X=[H,Y]_\theta$. Then for
every $p\in (1,\infty)$,
\begin{equation*}\label{eq:interpolation bound}
\mathfrak{P}_p(X)\leq \frac{2\max\{p,p/(p-1)\}}{1-\theta}.
\end{equation*}
\end{theorem}

\begin{remark}\label{rem:dependence on p}
{\em If $r\in (2,\infty)$ then the $O(\log n)$ term in Pisier's
inequality~\eqref{eq:classical pisier}, when $p=2$ and $X=\ell_r$,
can be replaced by $O(r)$, due to  the fact that
$\beta_2^+(\ell_r)\asymp r$ (this follows from Hitczenko's
work~\cite{Hit93}, as explained to us by Mark Veraar). This bound
also follows from Theorem~\ref{thm:interpolation}. At the same time,
an inspection of Talagrand's example in~\cite{Tal93} shows that this
term must be at least a constant multiple of $\log r$. Determining
the correct order of magnitude as $r\to \infty$ of the constant in
Pisier's inequality when $X=\ell_r$ remains an interesting open
problem.}
\end{remark}

\section{Duality}\label{sec:dual}
The dimension $n\in \N$ will be fixed from now on. For $p\in
[1,\infty]$ and a Banach space $X$, let $L_p(X)$ denote the
vector-valued Lebesgue space $L_p((\c,\mu),X)$. Thus $L_p(L_p(X))$
can be naturally identified with the space $L_p((\c\times
\c,\mu\times \mu),X)$.

For $f\in L_p(X)$ we denote its Fourier expansion by
$$
f=\sum_{A\subset\n} \hat{f}(A)W_A,
$$
where the Walsh function $W_A:\c\to \{-1,1\}$ corresponding to
$A\subseteq \n$ is given by $W_A(\e_1,\ldots,\e_n)=\prod_{i\in A}
\e_i$, and the Fourier coefficient $\hat{f}(A)\in X$ is given by
$\hat{f}(A)=\int_\c f(x)W_A(x)d\mu(x)$. Using this (standard)
notation, we have
$$
\forall\, i\in \n,\ \forall\, f\in L_p(X),\quad \partial_if=\sum_{\substack{A\subset\n\\i\in A}}\hat{f}(A)W_A,
$$
$$
\forall\, f\in L_p(X),\quad \Delta f=\sum_{A\subseteq\n} |A|\hat{f}(A)W_A,
$$
and
$$
\forall\, f\in L_p(X),\quad \Delta^{-1}f\eqdef \sum_{\substack{A\subset\n\\A\neq \emptyset}} \frac{1}{|A|}\hat{f}(A)W_A.
$$

The Rademacher projection of $f\in L_p(X)$ is defined as usual by
$$\Rad(f)\eqdef \sum_{i=1}^n \hat{f}(\{i\})W_{\{i\}}.$$ We denote below
$\Rad_X\eqdef \Rad(L_p(X))$ and
$\Rad_X^\perp\eqdef(I-\Rad)(L_p(X))$. The dual of
$(\Rad_X,\|\cdot\|_{L_p(X)})$ is naturally identified with the
quotient $L_q(X^*)/\Rad_{X^*}^\perp$, where $q=p/(p-1)$.


Define an operator $S:L_p(L_p(X))\to L_p(X)$ by
\begin{equation}\label{eq:def S}
\forall\, F\in L_p(L_p(X)),\quad S(F)\eqdef \sum_{j=1}^n \Delta^{-1}\partial_j\hat{F}(\{j\}).
\end{equation}
Using this notation, Theorem~\ref{thm:UMD+} is nothing more than the
following operator norm bound.
$$
\|S\|_{L_p(L_p(X))\to L_p(X)}\le \beta_q^+(X^*).
$$
The adjoint operator $S^*:L_q(X^*)\to L_q(L_q(X^*))$ is given by
$$
\forall\, g\in L_q(X^*),\ \forall\, \d\in \c,\quad S^*(g)(\d)=\sum_{j=1}^n\d_j\Delta^{-1}\partial_jg.
$$
Therefore Theorem~\ref{thm:UMD+} has the following equivalent dual
formulation.
\begin{theorem}[Dual formulation of
Theorem~\ref{thm:UMD+}]\label{thm:dual UMD+} Let $Z$ be a
$\mathrm{UMD}^+$ Banach space. Then for every $q\in (1,\infty)$ and
every $g\in L_q(Z)$ we have
$$
\left(\int_\c \Bigg\|\sum_{j=1}^n \d_j\Delta^{-1}\partial_j g\Bigg\|^q_{L_q(Z)}d\mu(\d)\right)^{1/q}
\le \beta_q^+(Z)\|g\|_{L_q(Z)}.
$$
\end{theorem}
Theorem~\ref{thm:dual UMD+}, and consequently also
Theorem~\ref{thm:UMD+}, will be proven in Section~\ref{sec:UMD+}.

Let $T$ be the restriction of $S$ to  $\Rad_{L_p(X)}$. Thus
$$
\mathfrak{P}_p^n(X)=\|T\|_{\Rad_{L_p(X)}\to L_p(X)}=\|T^*\|_{L_q(X^*)\to L_q(L_q(X^*))/\Rad_{L_q(X^*)}^\perp}.
$$
The adjoint $T^*:L_q(X^*)\to L_q(L_q(X^*))/\Rad_{L_q(X^*)}^\perp$ is
given by
$$
\forall\, g\in L_q(X^*),\ \forall\,\d\in\c,\quad T^*(g)= \sum_{j=1}^n\d_j\Delta^{-1}\partial_jg+\Rad_{L_q(X^*)}^\perp.
$$
Therefore Theorem~\ref{thm:generalized pisier}  has the following
equivalent dual formulation.
\begin{theorem}[Dual formulation of Theorem~\ref{thm:generalized
pisier}]\label{thm:dual gen pisier} Let $Z$ be a Banach space and
$q\in [1,\infty]$. Then for every $g\in L_q(Z)$ we have
\begin{multline*}
\inf_{\Phi\in \Rad_{L_q(Z)}^\perp}\left(\int_\c
\Bigg\|\Phi(\d)+\sum_{j=1}^n \d_j\Delta^{-1}\partial_j
g\Bigg\|^q_{L_q(Z)}d\mu(\d)\right)^{1/q}\\ \le
\left(\sum_{k=1}^{n}\frac{1}{k}\right)\|g\|_{L_q(Z)}.
\end{multline*}
\end{theorem}
Theorem~\ref{thm:dual gen pisier}, and consequently also
Theorem~\ref{thm:generalized pisier}, will be proven in
Section~\ref{sec:UMD+}. Since $[H,Y]_\theta^*=[H,Y^*]_\theta$
(see~\cite{BL76}), we also have the following equivalent dual
formulation of Theorem~\ref{thm:interpolation}.
\begin{theorem}[Dual formulation of
Theorem~\ref{thm:interpolation}]\label{thm:dual interpolation} Let
$H$ be a Hilbert space, $W$ a Banach space, and $\theta\in (0,1)$.
Set $Z=[H,W]_\theta$. Then for every $q\in (1,\infty)$ and $g\in
L_q(Z)$,
\begin{multline*}
\inf_{\Psi\in \Rad_{L_q(Z)}^\perp}\left(\int_\c
\Bigg\|\Psi(\d)+\sum_{j=1}^n \d_j\Delta^{-1}\partial_j
g\Bigg\|^q_{L_q(Z)}d\mu(\d)\right)^{1/q}\\ \leq
\frac{2\max\{q,q/(q-1)\}}{1-\theta} \|g\|_{L_q(Z)}.
\end{multline*}
\end{theorem}
Theorem~\ref{thm:dual interpolation}, and consequently also
Theorem~\ref{thm:interpolation}, will be proven in
Section~\ref{sec:proof interpolation}.

\section{Proof of Theorem~\ref{thm:dual UMD+}}\label{sec:UMD+}
Fix $q\in (1,\infty)$ and $g\in L_q(Z)$. Let $S_n$ denote the
symmetric group on $\{1,\ldots,n\}$. For  $\sigma\in S_n$ and $k\in
\{0,\ldots,n\}$  define $g_k^\sigma\in L_q(Z)$ by
\begin{multline}\label{eq:def g sigma}
g_k^\sigma(\e)\eqdef \sum_{A
\subseteq
\left\{\sigma^{-1}(1),\ldots,\sigma^{-1}(k)\right\}}\hat{g}(A)W_A(\e)
\\=\frac{1}{2^{n-k}}\sum_{\d_{\sigma^{-1}(k+1)},\ldots,\d_{\s^{-1}(n)}\in
\{-1,1\}} g\left(  \sum_{i=1}^k
\e_{\sigma^{-1}(i)}e_{\sigma^{-1}(i)}+\sum_{i=k+1}^n
\d_{\sigma^{-1}(i)}e_{\sigma^{-1}(i)}   \right),
\end{multline}
where here, and in what follows, $e_1,\ldots,e_n$ denotes the
standard basis of $\R^n$. Then $\{g_k^\sigma\}_{k=0}^n$ is a
$Z$-valued martingale with $g^\sigma_n=g$ and
$g^\s_0=\hat{g}(\emptyset)$, implying that
\begin{multline}\label{eq;each sigma}
\left(\int_\c\Bigg\|\sum_{k=1}^n\d_k\left(g_{k}^\s-g_{k-1}^\s\right)
\Bigg\|^{q}_{L_{q}(Z)}d\mu(\d)\right)^{1/q} \le \beta_{q}^+(Z)
\left\|g\right\|_{L_{q}(Z)}.
\end{multline}

In \eqref{eq;each sigma} we may replace $\{\delta_k\}_{k=1}^n$  by $\{\delta_{\sigma^{-1}(k)}\}_{k=1}^n$, since these two sequences of signs have the same joint distribution. Then we make the change of variable $j=\sigma^{-1}(k)$, so that $k=\sigma(j)$. Averaging the resulting inequality over $\sigma\in S_n$, and using
the convexity of the norm, we see that
\begin{multline}\label{eq:minus hat}
\left(\int_\c\Bigg\|\frac{1}{n!}\sum_{\s\in
S_n}\sum_{j=1}^n\d_j\left(g^\sigma_{\sigma(j)}
-g^\sigma_{\sigma(j)-1}\right)\Bigg\|^{q}_{L_{q}(Z)}d\mu(\d)\right)^{1/q}
\\
\le \beta_{q}^+(Z) \left\|g\right\|_{L_{q}(Z)}.
\end{multline}
It remains to note that for each $\d\in\c$ we have
\begin{align}\label{eq:permutation identity}
&\nonumber\frac{1}{n!}\sum_{\s\in S_n}\sum_{j=1}^n\d_j\left(g^\sigma_{\sigma(j)}-g^\sigma_{\sigma(j)-1}\right)\\&\nonumber=
\frac{1}{n!}\sum_{\s\in S_n}\sum_{j=1}^n\d_j
\sum_{\substack{\emptyset \subsetneq A\subseteq\{1,\ldots,n\}\\\max\sigma(A)= \sigma(j) }}\hat{g}(A)W_A
\\\nonumber&=\sum_{\substack{A\subseteq \{1,\ldots,n\}\\A\neq \emptyset}}\sum_{j\in A}
\d_j\frac{|\{\s\in S_n:\max\sigma(A)=\sigma(j) \}|}{n!} \hat{g}(A)W_A\\&=
\sum_{\substack{A\subseteq \{1,\ldots,n\}\\A\neq \emptyset}}\frac{\sum_{j\in A} \d_j}{|A|}\hat{g}(A)W_A\nonumber\\
&= \sum_{j=1}^n\d_j\Delta^{-1}\partial_jg.
\end{align}
Due to~\eqref{eq:minus hat} and~\eqref{eq:permutation identity} the
proof of Theorem~\ref{thm:dual UMD+} is complete. \qed

\section{Proof of Theorem~\ref{thm:dual gen pisier}}\label{sec:generalized pisier proof}

The following lemma introduces an auxiliary function which is a
variant of a similar function that was used by Pisier
in~\cite{Pis86}.

\begin{lemma}\label{lem:auxiliary}
Let $Z$ be a Banach space. Fix $n\in \N$, $q\in [1,\infty]$ and
$t\in (0,1)$. For $g\in L_q(Z)$ define $G_t\in L_q(L_q(Z))$ by
\begin{equation}\label{eq:def new interpolated}
G_t(\d)\eqdef \frac{1}{1-t}\sum_{A\subseteq
\{1,\ldots,n\}}\hat{g}(A)W_A\prod_{i\in
A}\left(t+(1-t)\d_i\right)-\frac{t^n}{1-t}g.
\end{equation}
Then
\begin{equation}\label{eq;rad of interpolated}
\Rad(G_t)(\d)=\sum_{\substack{A\subseteq \n\\A\neq\emptyset}} t^{|A|-1}\sum_{j\in A}\d_j\hat{g}(A)W_A,
\end{equation}
and
\begin{equation}\label{eq:norm of interpolated}
\|G_t\|_{L_q(L_q(Z))}\le \frac{1-t^n}{1-t}\|g\|_{L_q(Z)}.
\end{equation}
\end{lemma}
\begin{proof}
Identity~\eqref{eq;rad of interpolated} follows from~\eqref{eq:def
new interpolated} since for every $A\subseteq \n$,
$$
\Rad\left(\prod_{i\in A}\left(t+(1-t)\d_i\right)\right)=t^{|A|-1}(1-t)\sum_{j\in A}\d_j.
$$
To prove~\eqref{eq:norm of interpolated} observe that for every
$\e,\d\in \c$,
\begin{align}
&\nonumber(1-t)G_t(\d)(\e)
\\\label{eq:use def G_t}&=\sum_{A\subseteq \n}\hat{g}(A)W_A(\e)
\prod_{i=1}^n \left(t+(1-t)\d_i^{\1_A(i)}\right)-t^ng(\e)\\\nonumber
& = \sum_{A\subset\n}\hat{g}(A)W_A(\e)\sum_{B\subseteq \n} t^{|B|}(1-t)^{n-|B|}W_{A\setminus B}(\d)-t^ng(\e)\\
&= \sum_{B\subsetneq \n}t^{|B|}(1-t)^{n-|B|}\sum_{A\subset\n}\hat{g}(A)W_{A\cap B}(\e)
W_{A\setminus B}(\e\d)\nonumber\\
&=\sum_{B\subsetneq \n}t^{|B|}(1-t)^{n-|B|}g_B(\e,\d),\label{eq:expand H}
\end{align}
where in~\eqref{eq:use def G_t} we use~\eqref{eq:def new
interpolated} and in~\eqref{eq:expand H} for every $B\subset \n$ we
set
$$
g_B(\e,\d)\eqdef g\left(\sum_{j\in B}\e_je_j+\sum_{j\in \n\setminus B}\e_j\d_je_j\right).
$$
Since $g_B$ is equidistributed with $g$, it follows
from~\eqref{eq:expand H} that
\begin{equation*}
\frac{\|G_t\|_{L_q(L_q(Z))}}{\|g\|_{L_q(Z)}}\le \frac{1}{1-t}\sum_{B\subsetneq \n} t^{|B|}(1-t)^{n-|B|}
=\frac{1-t^n}{1-t}.\qedhere
\end{equation*}
\end{proof}
\begin{proof}[Proof of Theorem~\ref{thm:dual gen pisier}] Observe
that for every $\d\in \c$ we have
\begin{eqnarray}\label{eq:integrated identity}
\sum_{j=1}^n\d_j\Delta^{-1}\partial_j g&=&\sum_{\substack{A\subseteq
\n\\A\neq \emptyset}} \frac{1}{|A|}\sum_{j\in A} \d_j\hat{g}(A)W_A
\nonumber\\&=&\nonumber \sum_{\substack{A\subseteq \n\\A\neq
\emptyset}}\left(\int_0^1t^{|A|-1}dt\right)\sum_{j\in A}
\d_j\hat{g}(A)W_A\\&\stackrel{\eqref{eq;rad of
interpolated}}{=}&\Rad\left(\int_0^1G_t(\d)dt\right).
\end{eqnarray}
It follows that if we set
\begin{equation}\label{eq:def Phi}
\Phi\eqdef \int_0^1G_tdt-\Rad\left(\int_0^1G_tdt\right).
\end{equation}
then $\Phi\in \Rad_{L_q(Z)}^\perp$ and
\begin{multline*}
\left(\int_\c \Bigg\|\Phi(\d)+\sum_{j=1}^n \d_j\Delta^{-1}\partial_j
g\Bigg\|^q_{L_q(Z)}d\mu(\d)\right)^{1/q}\\
\stackrel{\eqref{eq:integrated identity}}{=}
\Bigg\|\int_0^1G_tdt\Bigg\|_{L_q(L_q(Z))}\stackrel{\eqref{eq:norm of
interpolated}}{\le}
\left(\int_0^1\frac{1-t^n}{1-t}dt\right)\|g\|_{L_q(Z)}.
\end{multline*}
It remains to note that  $
\int_0^1\frac{1-t^n}{1-t}dt=\sum_{k=0}^{n-1}\int_0^1t^kdt=\sum_{k=1}^n\frac{1}{k}.
$
\end{proof}

\section{Proof of Theorem~\ref{thm:dual
interpolation}}\label{sec:proof interpolation} For $t\in (0,1)$
define a linear operator $V_t:L_q(Z)\to L_q(L_q(Z))$ by
\begin{multline}\label{eq:def V_t}
 V_t(g)(\d)\eqdef G_t(\d)-\hat{G_t}(\emptyset)
\\
\stackrel{\eqref{eq:def new
interpolated}}{=}\frac{1}{1-t}\sum_{A\subset\n}
\hat{g}(A)W_A\left(\prod_{i\in A}(t+(1-t)\d_i)-t^{|A|} \right).
\end{multline}

\begin{lemma}\label{lem:Hilbert case}
Let $H$ be a Hilbert space. Then for every $t\in (0,1)$,
\begin{equation}\label{eq:root behavior}
\|V_t\|_{L_2(H)\to L_2(L_2(H))}{\leq}  \frac{1}{\sqrt{1-t^2}}\le \frac{1}{\sqrt{1-t}}.
\end{equation}
\end{lemma}
\begin{proof}
Observe that for every $A\subset\n$ we have
\begin{multline}\label{eq:power diff}
\int_\c \left(\prod_{i\in
A}(t+(1-t)\d_i)-t^{|A|}\right)^2d\mu(\d)\\=\sum_{B\subsetneq
A}t^{2|B|}(1-t)^{2(|A|-|B|)}
=\left(t^2+(1-t)^2\right)^{|A|}-t^{2|A|}.
\end{multline}
It follows from~\eqref{eq:def V_t}, \eqref{eq:power diff}, and the
orthogonality of $\{W_A\}_{A\subseteq\n}$, that
\begin{equation}\label{eq:supremum}
\|V_t\|_{L_2(H)\to L_2(L_2(H))} = {\max_{a\in\{1,\ldots,n\}} }\frac{\sqrt{\left(t^2+(1-t)^2\right)^{a}-t^{2a}}}{1-t}.
\end{equation}
Now, for every $a\in \{1,\ldots,n\}$ and $t\in (0,1)$ we have
\begin{multline}\label{eq:geometric sum}
  \left(t^2+(1-t)^2\right)^{a}-t^{2a}= (1-t)^2\sum_{k=0}^{a-1}\left(t^2+(1-t)^2\right)^{a-1-k}t^{2k}\\
  \le (1-t)^2\sum_{k=0}^{a-1} t^{2k}=(1-t)^2\frac{1-t^{2a}}{1-t^2}\le \frac{1-t}{1+t},
\end{multline}
where in the first inequality of~\eqref{eq:geometric sum} we used
the  estimate $t^2+(1-t)^2\leq 1$, which holds for every $t\in
[0,1]$. The desired estimate~\eqref{eq:root behavior} now follows
from a substitution of~\eqref{eq:geometric sum}
into~\eqref{eq:supremum}.
\end{proof}


\begin{lemma}\label{lem:interpolation t bound}
Let $H$ be a Hilbert space and let $W$ be a Banach space. Fix
$\theta\in (0,1)$ and $q\in (1,\infty)$. Set $Z=[H,W]_\theta$. Then
for every $t\in (0,1)$ we have
\begin{equation}\label{eq:gained interpolated}
\left\|V_t\right\|_{L_q(Z)\to L_q(L_q(Z))}\le \frac{2}{(1-t)^{1-{(1-\theta)}\min\{1/q,1-1/q\}}}.
\end{equation}
\end{lemma}

\begin{proof}
For every $r\in [1,\infty]$ we have
\begin{multline*}
\|V_t(g)\|_{L_r(L_r(W))}\stackrel{\eqref{eq:def V_t}}{=}
\left\|G_t-\hat{G_t}(\emptyset)\right\|_{L_r(L_r(W))}\\\le
2\|G_t\|_{L_r(L_r(W))} \stackrel{\eqref{eq:norm of
interpolated}}{\le} \frac{2\|g\|_{L_r(W)}}{1-t}.
\end{multline*}
Consequently,
\begin{equation}\label{eq:trivial t}
\forall\, r\in [1,\infty],\quad \left\|V_t\right\|_{L_r(W)\to L_r(L_r(W))}\leq \frac{2}{1-t}.
\end{equation}
If $q\in [2,\infty)$ then we interpolate (see~\cite{BL76})
between~\eqref{eq:root behavior} and~\eqref{eq:trivial t} with $W=H$
and $r=\infty$. If $q\in (1,2]$ then we interpolate
between~\eqref{eq:root behavior} and~\eqref{eq:trivial t} with $W=H$
and $r=1$. The norm bound thus obtained implies the estimate
\begin{equation}\label{eq:L_qH}
\forall\, q\in (1,\infty),\quad \left\|V_t\right\|_{L_q(H)\to L_q(L_q(H))}\leq \frac{2}{(1-t)^{\max\{1/q,1-1/q\}}}.
\end{equation}
Finally, interpolation between~\eqref{eq:L_qH} and~\eqref{eq:trivial
t} with $r=q$ gives the desired norm bound~\eqref{eq:gained
interpolated}.
\end{proof}

\begin{proof}[Proof of Theorem~\ref{thm:dual interpolation}]
By~\eqref{eq:def V_t} we have $\Rad(V_t(g))=\Rad(G_t)$. Therefore,
analogously to~\eqref{eq:def Phi}, if we set
$$
\Psi\eqdef  \int_0^1V_t(g)dt-\Rad\left(\int_0^1G_tdt\right)=\int_0^1V_t(g)dt-\Rad\left(\int_0^1V_t(g)dt\right),
$$
then $\Psi\in \Rad_{L_q(Z)}^\perp$ and by~\eqref{eq:integrated
identity} for every $\d\in \c$ we have
\begin{equation}\label{eq:Psi identity}
\Psi(\d)+\sum_{j=1}^n\d_j\Delta^{-1}\partial_j g=\int_0^1V_t(g)(\d)dt.
\end{equation}
Hence,
\begin{eqnarray*}
&&\!\!\!\!\!\!\!\!\!\!\!\!\!\!\!\!\!\!\!\!\!\!\!\!\!\!\!\!\!\!\!\!\!\!\!\!\!\!\!\!\!\!\!\!\!\!\!\!\!\!\!\!\!\!\!\!\!\!\left(\int_\c \Bigg\|\Psi(\d)+\sum_{j=1}^n \d_j\Delta^{-1}\partial_j
g\Bigg\|^q_{L_q(Z)}d\mu(\d)\right)^{1/q}\\
&\stackrel{\eqref{eq:Psi identity}\wedge\eqref{eq:gained interpolated}}{\leq}&
\int_0^1\frac{2\|g\|_{L_q(Z)}}{(1-t)^{1-(1-\theta)\min\{1/q,1-1/q\}}}dt\\&=&\frac{2\|g\|_{L_q(Z)}}{(1-\theta)\min\{1/q,1-1/q\}}.
\end{eqnarray*}
This is precisely the assertion of Theorem~\ref{thm:dual
interpolation}.
\end{proof}

\section{Enflo type in uniformly smooth Banach spaces}\label{sec:enflo} A Banach space $X$ has Rademacher type $p\in [1,2]$ (see e.g.~\cite{Mau03}) if there exists $T_R\in (0,\infty)$ such that for all $n\in \N$ and all $x_1,\ldots,x_n\in X$,
\begin{equation}\label{eq:type def}
\int_\c \Bigg\|\sum_{j=1}^n \e_j x_j\Bigg\|^pd\mu(\e)\le T_R^p\sum_{j=1}^n \|x_j\|^p.
\end{equation}
$X$ has Enflo type $p$ (see~\cite{Enf76,BMW86,Pis86,NS02}) if there
exists $T_E\in (0,\infty)$ such that for all $n\in \N$ and all
$f:\c\to X$,
\begin{equation}\label{eq:enf type def}
\int_\c \frac{\left\|f(\e)-f(-\e)\right\|^p}{2^p}d\mu(\e)\le T_E^p\sum_{j=1}^n \|\partial_j f\|^p_{L_p(X)}.
\end{equation}

By considering the function $f(\e)=\sum_{j=1}^n\e_j x_j$ one sees
that~\eqref{eq:type def} is a special case of~\eqref{eq:enf type
def}. It is a long-standing open problem~\cite{Enf76} whether,
conversely, \eqref{eq:type def} implies \eqref{eq:enf type def}. A
crucial feature of~\eqref{eq:enf type def} is that it is a purely
metric condition (thus one can define when a {\em metric space} has
Enflo type $p$), while~\eqref{eq:type def} is a linear condition.
See~\cite{MN07-scaled} for a   purely metric condition (which is
more complicated than, but inspired by, Enflo type) that is known to
be equivalent to Rademacher type.

Observe that if~\eqref{eq:type def} holds then it follows
from~\eqref{eq:def extended pisier} that for every
$f_1,\ldots,f_n:\c\to X$,
\begin{equation}\label{eq:use type in gen pisier}
\Bigg\|\sum_{j=1}^n \Delta^{-1}\partial_j f_j\Bigg\|_{L_p(X)}\le T_R\mathfrak{P}_p^n(X)\left(\sum_{j=1}^n \|f_j\|_{L_p(X)}^p\right)^{1/p}.
\end{equation}
The special case $f_j=\partial_j f$ shows that~\eqref{eq:use type in
gen pisier} implies~\eqref{eq:enf type def} with $$T_E\le
T_R\mathfrak{P}_p^n(X).$$  For this reason it is worthwhile to
investigate~\eqref{eq:use type in gen pisier} on its own right.

Let $\mathfrak{Q}_p^n(X)$ be the infimum over those $\mathfrak{Q}\in
(0,\infty)$ such that every  $f_1,\ldots,f_n:\c\to X$ satisfy
\begin{equation}\label{eq:def Q const}
\Bigg\|\sum_{j=1}^n \Delta^{-1}\partial_j f_j\Bigg\|_{L_p(X)}\le \mathfrak{Q}\left(\sum_{j=1}^n \|f_j\|_{L_p(X)}^p\right)^{1/p}.
\end{equation}
We also set
$$
\mathfrak{Q}_p(X)\eqdef \sup_{n\in \N} \mathfrak{Q}_p^n(X).
$$
By duality, $\mathfrak{Q}_p^n(X)$ equals the infimum over those
$\mathfrak{Q}\in (0,\infty)$ for which every $g\in L_q(X^*)$
satisfies
\begin{equation}\label{eq:dual Q}
\left(\sum_{j=1}^n \left\|\Delta^{-1}\partial_j g\right\|_{L_q(X^*)}^q\right)^{1/q}\le \mathfrak{Q}\|g\|_{L_q(X^*)}.
\end{equation}

Letting $S_X=\{x\in X:\ \|x\|=1\}$ denote the unit sphere of $X$,
recall that the { modulus of uniform convexity} of $X$ is defined
for $\e\in [0,2]$ as
\begin{equation*}\label{def:convexity}
\delta_X(\e)=\inf\left\{ 1-\frac{\|x+y\|}{2}:\
x,y\in S_X,\ \|x-y\|=\e\right\}.
\end{equation*}
The {modulus of uniform smoothness} of $X$ is defined for $\tau\in
(0,\infty)$ as
\begin{equation*}\label{eq:def smoothness}
\rho_X(\tau)\eqdef \sup \left\{\frac{\|x+\tau y\|{+}\|x-\tau
y\|}{2}-1:\ x,y\in S_X\right\}.
\end{equation*}
These moduli  relate to each other via the following classical
duality formula of Lindenstrauss~\cite{Lin63}.
\begin{equation}\label{eq:lindenstrauss duality}
\delta_{X^*}(\e)=\sup\left\{\frac{\tau\e}{2}-\rho_X(\tau):\ \tau\in
[0,1]\right\}.
\end{equation}

\begin{theorem}\label{thm:Q with convexity}
For every $K,p\in (1,\infty)$ there exists $C(K,p)\in (0,\infty)$
such that if $X$ is a Banach space that satisfies $\rho_X(\tau)\le
K\tau^p$ for all $\tau\in (0,\infty)$, then $\mathfrak{Q}_p(X)\le
C(K,p)$.
\end{theorem}

\begin{proof} We shall use here the notation introduced in the proof of Theorem~\ref{thm:dual UMD+} (Section~\ref{sec:UMD+}). It follows from~\eqref{eq:lindenstrauss duality} that $\d_{X^*}(\e){\gtrsim}_{K,p} \e^q$ for every $\e\in [0,2]$ (here, and it what follows, the notation $\lesssim_{K,p}$ suppresses constant factors that may depend only on $K$ and $p$). Hence, for $g\in L_q(X^*)$ and $\sigma\in S_n$, since $\{g_k^\sigma\}_{k=0}^n$, as defined in~\eqref{eq:def g sigma}, is an
$X^*$-valued martingale, it follows from Pisier's martingale
inequality~\cite{Pis75} that
\begin{equation}\label{eq:to average end}
\left(\sum_{k=1}^n \|g_k^\sigma-g_{k-1}^\sigma\|_{L_q(X^*)}^q\right)^{1/q}\lesssim_{K,p} \|g\|_{L_q(X^*)}.
\end{equation}
By reindexing~\eqref{eq:to average end} with $k=\sigma(j)$, averaging over $\s\in S_n$, and using
the convexity of the norm, we obtain the estimate
\begin{equation}\label{eq:averaged end}
\left(\sum_{j=1}^n \Bigg\|\frac{1}{n!}\sum_{\s\in S_n}\left(g^\sigma_{\sigma(j)}-g^\sigma_{\sigma(j)-1}\right)\Bigg\|_{L_q(X^*)}^q\right)^{1/q}
\lesssim_{K,p}  \|g\|_{L_q(X^*)}.
\end{equation}
Arguing as in~\eqref{eq:permutation identity}, for every $j\in \n$
we have the identity
\begin{multline}\label{eq;perumtation identity end}
\frac{1}{n!}\sum_{\s\in
S_n}\left(g^\sigma_{\sigma(j)}-g^\sigma_{\sigma(j)-1}\right)=\frac{1}{n!}\sum_{\s\in
S_n}\sum_{\substack{\emptyset \subsetneq
A\subseteq\{1,\ldots,n\}\\\max
\sigma(A)=\s(j)}}\hat{g}(A)W_A\\=\sum_{\substack{A\subset\n\\ j\in
A}}\frac{|\{\s\in S_n:\
\max\s(A)=\s(j)\}|}{n!}\hat{g}(A)W_A=\Delta^{-1}\partial_j g.
\end{multline}
Consequently, \eqref{eq:averaged end} combined
with~\eqref{eq;perumtation identity end} imply that~\eqref{eq:dual
Q} holds true with $\mathfrak{Q}\lesssim_{K,p} 1$. This concludes
the proof of Theorem~\ref{thm:Q with convexity}.
\end{proof}
\begin{remark}
{\em It follows from~\cite[Sec.~6]{KN06} that a Banach space $X$
satisfying the assumption of Theorem~\ref{thm:Q with convexity} has
Enflo type $p$. Theorem~\ref{thm:Q with convexity} can be viewed as
a generalization of this fact to yield the inequality~\eqref{eq:def
Q const}. In~\cite{NPSS06} it was shown that any Banach space
satisfying the assumption of Theorem~\ref{thm:Q with convexity}
actually has K. Ball's Markov type $p$ property~\cite{Bal92}, a
property which is a useful strengthening of Enflo type $p$.}
\end{remark}

\subsection*{Acknowledgements}
T. H. is supported in part by the European Union through the ERC
Starting Grant ``Analytic probabilistic methods for borderline
singular integrals", and by the Academy of Finland through projects
130166 and 133264. A.~N. is supported in part by NSF grant
CCF-0832795, BSF grant
 2010021, the Packard Foundation and the Simons Foundation.
 Part of this work was completed while A. N. was visiting Universit\'e Pierre et Marie Curie, Paris, France.

\end{document}